\documentclass[11pt]{amsart}
\usepackage{esint, amssymb}
\usepackage{stmaryrd}
\usepackage{fancyhdr}
\usepackage{enumerate}
\usepackage{mathtools}
\usepackage{tikz}
\usetikzlibrary{automata,positioning}
\usepackage{bbm}
\usepackage{blkarray}
\usepackage{xfrac}

\usepackage[text={420pt,650pt},centering]{geometry}

\usepackage{xcolor}
\usepackage{graphicx}
\usepackage{MnSymbol}
\usepackage[colorlinks=true, pdfstartview=FitV, linkcolor=blue, citecolor=blue, urlcolor=blue,pagebackref=false]{hyperref}
\usepackage{microtype}

\usepackage{bm}
\usepackage{scalerel} 
\usepackage{dsfont}
\usepackage[font={footnotesize}]{caption}

\usepackage{verbatim}

\usepackage{scalerel}[2014/03/10]
\usepackage[usestackEOL]{stackengine}
\def\dashint{\,\ThisStyle{\ensurestackMath{%
  \stackinset{c}{.2\LMpt}{c}{.5\LMpt}{\SavedStyle-}{\SavedStyle\phantom{\int}}}%
  \setbox0=\hbox{$\SavedStyle\int\,$}\kern-\wd0}\int}
  
\newcommand{\dashsum}{%
    \mathop{
        \mathchoice
        {\ooalign{$\displaystyle\sum$\cr\vphantom{$\displaystyle\sum$}\hidewidth$\displaystyle\kern 0.2ex\vcenter{\hrule height 0.8pt width 0.9em}\kern 0.2ex$\hidewidth}}
        {\ooalign{$\textstyle\sum$\cr\vphantom{$\textstyle\sum$}\hidewidth$\textstyle\kern 0.2ex\vcenter{\hrule height 0.6pt width 0.75em}\kern 0.2ex$\hidewidth}}
        {\ooalign{$\scriptstyle\sum$\cr\vphantom{$\scriptstyle\sum$}\hidewidth$\scriptstyle\kern 0.15ex\vcenter{\hrule height 0.4pt width 0.6em}\kern 0.15ex$\hidewidth}}
        {\ooalign{$\scriptscriptstyle\sum$\cr\vphantom{$\scriptscriptstyle\sum$}\hidewidth$\scriptscriptstyle\kern 0.15ex\vcenter{\hrule height 0.3pt width 0.5em}\kern 0.15ex$\hidewidth}}
    }
}

\fancypagestyle{first}{%
  \fancyhf{}
  \headrulewidth{0pt}

}

\newcommand{\N}{\mathbb N}

\newcommand{\R}{\mathbb R}

\newcommand{\Rd}{\mathbb R^d}
\newcommand{\Zd}{\mathbb{Z}^d}

\newcommand{\ep}{\varepsilon}

\renewcommand{\a}{\mathbf{a}}

\newcommand{\comp}[1]{#1^\text{c}}

\newcommand{\mc}[1]{\mathcal{#1}}

\definecolor{officegreen}{rgb}{0.0, 0.5, 0.0}

\newcommand{\hyp}{{\mathrm{hyp}}}
\newcommand{\kin}{{\mathrm{kin}}}

\DeclarePairedDelimiter{\abs}{\lvert}{\rvert}
\DeclarePairedDelimiter{\inp}{(}{)}

\DeclarePairedDelimiter{\insb}{\{}{\}}

\DeclarePairedDelimiter{\inprod}{\langle}{\rangle}
\DeclarePairedDelimiter{\norm}{\|}{\|}
\DeclarePairedDelimiter{\snorm}{\llbracket}{\rrbracket}

\DeclareMathOperator{\tr}{tr}

\newtheorem{thm}{Theorem}
\numberwithin{thm}{section}

\numberwithin{lem}{section}

\numberwithin{cor}{section}
\newtheorem{prop}{Proposition}
\numberwithin{prop}{section}
\numberwithin{equation}{section}

\newcommand{\tref}[1]{Theorem~\ref{t.#1}}
\newcommand{\pref}[1]{Proposition~\ref{p.#1}}

\newcommand{\cref}[1]{Corollary~\ref{c.#1}}

\newcommand{\eref}[1]{(\ref{e.#1})}

\theoremstyle{definition}

\numberwithin{defn}{section}

\numberwithin{ex}{section}

\theoremstyle{remark}

\numberwithin{rem}{section}

\hypersetup{
    colorlinks,
    citecolor=black,
    filecolor=black,
    linkcolor=black,
    urlcolor=black
}

\begin{document}
\title[Nash-Aronson Estimate for the Linear Kinetic Fokker-Planck equation]{Nash-Aronson Estimate for the Linear Kinetic Fokker-Planck equation}

\begin{abstract}
We prove a Nash-Aronson-type upper bound on the fundamental solution of the linear kinetic Fokker Planck equation with friction term, distinguishing two regimes. For long times, we derive a Gaussian upper bound matching the classical parabolic estimate, which reflects the averaging of the velocity variable that occurs in this regime. For short times, the fundamental solution is governed by that of the constant-coefficient Kolmogorov equation, with the friction and potential terms negligible. 
\end{abstract}

\author[P. Gaddy]{Philip Gaddy}
\address[P. Gaddy]{Courant Institute of Mathematical Sciences, New York University, 251 Mercer St., New York, NY 10012}
\email{philip.gaddy@courant.nyu.edu}

\date{\today}
\maketitle

\section{Introduction}
\subsection{Motivation and Summary of Results}
We study the forced kinetic Fokker-Planck equation
\begin{equation}
\label{e.hypoeq}
\partial_t f-\nabla_v\cdot\a\nabla_v f + v\cdot\a\nabla_v f -v\cdot \nabla_x f + \nabla H \cdot \nabla_v f = 0, \quad\text{on }(0,\infty)\times\Rd\times\Rd,
\end{equation}
where the diffusion matrix $\a(t,x,v)$ is uniformly elliptic in $v$, essentially bounded, and satisfies
\begin{equation}
\label{e.coefbnd}
\abs{v\cdot\nabla_v\cdot \a} \leq \Lambda\,.
\end{equation}
while the potential $H$ is bounded and differentiable. This is the Fokker-Planck equation associated with the Langevin dynamics
\begin{equation}
\label{e.hypomarkoveq}
\begin{cases}
dX_t = V_tdt \\
dV_t \,\,= (-\a V_t + \nabla H(X_t))dt + \sqrt{2}\a^{\sfrac{1}{2}}dW_t\,,
\end{cases}
\end{equation}
so $(X_t,V_t)$ is a hypoelliptic diffusion combining transport in $(x,v)$, linear friction in $v$, and noise only in the velocity variable. Our goal is to obtain Gaussian upper bounds of Nash-Aronson type for the Green's function $P(t,x,v,s,y,w)$ of this equation. In doing so, we demonstrate that the short-time Kolmogorov scaling develops into the long-time diffusive regime as time increases. Roughly speaking we prove that, for
$$P(t,x,v,s,y,w)\lesssim (t-s)^{-2d}\exp(-c\,d_{\hyp}^2/(t-s)) $$
for small $t-s$, where $d_{\hyp}$ is the anisotropic distance associated with the vector fields $\partial_v$ and $v\cdot \nabla_x$, that is 
$$d_{\hyp}(t,x,v,s,y,w)^2 = \abs{(t-s)^{-1}(x-y) +w}^2+ \abs{v-w}^2\,.$$
Additionally, we show that
$$P(t,x,v,s,y,w)\lesssim (t-s)^{-d/2}\exp(-c\,|x-y|^2/(t-s)) $$
for large $t-s$,
with constants depending only on the ellipticity and growth bounds on $a$ and $H$. In particular, integrating in $(x,v)$ shows exponential convergence of solutions of (1.1) to the equilibrium state, giving a PDE-based proof of hypocoercivity in this general setting.

More specifically, we prove that
\begin{thm}
\label{t.nasharonson}
Let $P(t,x,v,s,y,w)$ be the Green function for \eref{hypoeq}. Then there exists a constant $C>0$ such that for $A> \frac{256(4+2\norm{\nabla H}^2_\infty)}{3}$ and for all $t,s>0$ and  $x,y,v,w\in\Rd$,
$$P(t,x,v,s,y, w) \leq C (t-s)^{-2d}\exp\inp*{-\tfrac{3\abs{(x-y) +(t-s)w}^2}{A(t-s)^3} - \tfrac{32\abs{v-w}^2 +32\abs{w}^2}{3A(t-s)}} +C (t-s)^{-\frac{d}{2}}\exp\inp*{-\tfrac{\abs{x-y}^2}{8A(t-s)}} \,.$$
\end{thm}

\subsection{Related work.}
There is a substantial literature on fundamental solutions and Gaussian estimates for degenerate Kolmogorov-Fokker-Planck operators. This line of investigation originates with the classical Kolmogorov equation with constant coefficients:
\begin{equation}
\label{e.kolmeq}
\partial_t f -\Delta_v f - v\cdot\nabla_xf = 0
\end{equation}
Kolmogorov \cite{K} and H\"ormander \cite{H} not only computed the explicit Gaussian kernel for this equation but also connected these explicit formulas to the notions of hypoellipticity, thereby establishing the foundational framework for later developments.

When the coefficients become variable, the problem takes on a new layer of complexity. Pascucci and Polidoro \cite{Pasc1, Pasc2} extended the theory by developing Aronson-type Gaussian upper bounds and Nash estimates for a broad class of ultraparabolic operators in divergence form, working under H\"ormander-type structural conditions.  Lanconelli and Pascucci \cite{Lanc} later extended Nash-Aronson bounds to non-homogeneous Kolmogorov equations, thus widening the applicability of these classical estimates to more intricate kinetic models.

Degenerate Kolmogorov equations with rough coefficients present another significant challenge. Anceschi and Rebucci \cite{AR} addressed this by constructing an explicit fundamental solution and proving Gaussian bounds, taking advantage of the dilation-invariant structure inherent to the associated Lie groups. Building upon this foundation, Auscher et al. \cite{Aus} recently introduced a general $H^1_\hyp$ well-posedness theory, which provides a robust analytic framework for Kolmogorov-Fokker-Planck equations with rough coefficients, and in turn yields Gaussian upper estimates.

Similar results have been proven from the probabilistic perspective. Zhang \cite{FS} employed the tools of Malliavin calculus to establish the existence of smooth densities, which serve as fundamental solutions, for kinetic Fokker-Planck operators exhibiting anisotropic nonlocal dissipativity. Leli\`evre, Ramil, and Reygner \cite{Lelievre} studied the absorbed Langevin process in cylindrical domains, proving the existence and regularity of its transition density and deriving an explicit Gaussian upper bound.

Further investigations have explored the effects of boundaries and geometric constraints in kinetic Fokker-Planck equations. Existence and regularity results in this context have been obtained by Carrillo \cite{Car} for Vlasov-Poisson-Fokker-Planck systems, by Nier \cite{Nierbook} in a geometric setting, and more recently by Silvestre \cite{S}, Zhu \cite{Z}, Bernard \cite{Ber}, among others. Finally, the long-time behavior of kinetic Fokker-Planck dynamics, often referred to as hypocoercivity, has been analyzed through both probabilistic and analytic techniques. Probabilistic approaches have been developed by Rey-Bellet and Thomas \cite{RBT}, while analytic treatments can be found in the works of Herau \cite{H1}, Villani \cite{V}, Mischler and Mouhot \cite{Misc}, Baudoin \cite{baudoin}, and others, including the contributions of Hwang et al. \cite{Hwang}.

In contrast to these works, we focus on the forced kinetic Fokker-Planck equation (1.1) with a general time and space dependent diffusion matrix
$\a(t,x,v)$ satisfying the mild growth constraint 
$$\abs{v\cdot\nabla_v\cdot \a}\leq\Lambda,$$
 and a bounded, differentiable potential $H$. Our approach is variational: we work in the weighted kinetic Sobolev spaces introduced by Albritton-Armstrong-Mourrat-Novack\cite{AAMN} and adapt the classical Nash-Aronson method to this hypoelliptic setting by exploiting an extended energy form inspired by the constant-coefficient Kolmogorov kernel.

\subsection{Explanation of the Main Result.}
This main result is formulated as a single bound, but is most effectively thought of as two separate estimates in differing time regimes. In section 3, when proving this result, we first prove that for $t-s\geq 1$ and $A$ as above,
\begin{equation}
\label{e.longtimebound}
P(t,x,v,s,y, w) \leq C (t-s)^{-\frac{d}{2}}\exp\inp*{-\frac{\abs{x-y}^2}{8A(t-s)}}\, ,
\end{equation}
and then prove that
\begin{equation}
\label{e.shorttimebound}
P(t,x,v,s,y, w) \leq C (t-s)^{-2d}\exp\inp*{-\frac{3\abs{(x-y) +(t-s)w}^2}{A(t-s)^3} - \frac{32\abs{v-w}^2 +32\abs{w}^2}{3A(t-s)}}\, ,
\end{equation}
for $t-s \leq 1$.

This approach mirrors the expected homogenization result that inspired this paper, especially in the long time regime. In this regime, the friction terms, given by the operator $v\cdot\a\nabla_v + \nabla H \cdot\nabla_v$, contribute on the same order as the other components of the equation. We demonstrate that, after averaging over the $v$-variable, the fundamental solution behaves like the heat kernel, which is consistent with the homogenization of the associated stochastic differential equation (SDE), as seen in \cite{HP1, HP2}. In this setting, weighting the energy estimates by a Gaussian in the velocity variable becomes a key step, enabling us to obtain estimates analogous to the standard parabolic Nash-Aronson upper bound. 

This builds upon the ideas presented in \cite{AAMN} to prove the existence of solutions to equations on the form \eref{hypoeq}. The Gaussian weight in the $v$ variable and the definition of the $H^1_\hyp$ space prove to be the correct formulations to replicate the self-adjointedness properties of more standard parabolic equations in this more general case. In fact, the proof of \eref{longtimebound} is similar to the proof of the Nash-Aronson upper bound for parabolic equations. 

A notable aspect of the long-time bound is the absence of any explicit $v$-dependence on the right-hand side. This is due, in part, to the way we define the fundamental solution. Specifically, when constructing a solution to \eref{hypoeq} with initial data $g\in L^2(\Rd; L^2_\gamma)$, the convolution is taken with respect to the measure $dxd\gamma(v)$, where we define the $\gamma$ measure in the following section, which incorporates the $v$-dependence indirectly. In effect, our estimate implies that a solution to \eref{hypoeq} with initial data $g$ can be bounded by the scaled heat kernel given in \eref{longtimebound}, convolved with the Gaussian average of $g$. The influence of $v$ is therefore captured through the Gaussian measure, which aligns our expectation derived from the expected homogenization behavior for hypoelliptic equations.

Over shorter time scales, when the support of the initial data for the solution to \eref{hypoeq} is sufficiently small, we demonstrate that the friction terms act in a more perturbative manner. In this regime, the diffusion and transport components dominate, defining the primary component of the equation. Consequently, the fundamental solution closely resembles that of equation \eref{kolmeq}.

To leverage this observation in proving \eref{shorttimebound}, we redefine our notion of weak solution, omitting the Gaussian weight and instead requiring some decay in the $v$-variable within the function itself. After making this adjustment, we then re-derive several functional inequalities used in proving \eref{longtimebound}. Once we establish the framework for this new notion of solution, proving \eref{shorttimebound} proceeds, with some additional computations, similarly to the long-time case. Notably, this is the only point in the proof where \eref{coefbnd} is applied, as we use this bound to control one of the error terms that arises as a consequence of this perturbative approach.

\section{Functional Preliminaries and Inequalities}
\subsection{Long time function spaces and weak solutions}
We begin this section by discussing the functional framework that we will use when proving the long time estimate in \tref{nasharonson}. In this regime, we mostly rely on the functional analysis framework for solutions of \eref{hypoeq} developed throughout \cite{AAMN}. As most of the function spaces and notation used are non standard, we take this subsection to review the crucial parts of this theory.

To begin, we define the measures
\begin{equation*}
\left\{  
\begin{aligned}
& d\sigma(x) := \exp\left( -H(x) \right) \,dx\,, \\
& d\gamma(v):= (2\pi)^{-\frac{d}{2}} e^{ -\frac{\abs{v}^2}{2}} \,dv\,, \\
& dm(x,v) := d\sigma(x) d\gamma(v)\,, \\
& dm_t(t,x,v) := dtd\sigma(x) d\gamma(v)\,.
\end{aligned}
\right.
\end{equation*}
Note that, since $H$ is bounded, $\sigma$ is equivalent to the Lebesgue measure. More precisely, if $\abs{H}\leq \Lambda$, there are positive constants $C(\Lambda), c(\Lambda) <\infty$ such that
$$c(\Lambda) \int \abs{f} \,dx \leq \int \abs{f}\,d\sigma \leq C(\Lambda) \int \abs{f}\,dx\,.$$
For $V\subset \R\times\Rd$, we define the space $H^1_\kin(V)$ by
\begin{equation*}  
H^1_\kin(V) := \insb*{ f \in L^2\left(V;H^1_\gamma\right) \ : \ \partial_t f -v \cdot \nabla_x f \in L^2(V;H^{-1}_\gamma)}\,,
\end{equation*}
and it is a Banach space with respect to the norm
\begin{equation*}
\norm{f}_{H^1_\kin(V)} ^2
:= \norm{f}_{L^2(V;H^1_\gamma)}^2
+ 
\left\| \partial_t f - v \cdot \nabla_x f \right\|_{L^2(V;H^{-1}_\gamma)}^2 \,.
\end{equation*}
We also define the seminorm
\begin{equation*}
\snorm{f}^2_{H^1_\kin(V)} 
:=
\norm{\nabla_v f}_{L^2(V;L^2_\gamma)}^2
+ 
\norm{ \partial_t f - v \cdot \nabla_x f }_{L^2(V;H^{-1}_\gamma)}^2 
\end{equation*}
so that we have $\left\| f \right\|_{H^1_\kin(V)}^2 = \left\llbracket f \right\rrbracket_{H^1_\kin(V)}^2 + \left\| f \right\|_{L^2(V;L^2_\gamma)}^2$.

We say that $f \in H^1_\kin(V)$ is a \emph{weak solution} of the equation
\begin{equation*}
\partial_t f-\nabla_v\cdot\a\nabla_v f + v \cdot \a\nabla_v f - v \cdot \nabla_x f + \nabla H(x) \cdot \nabla_v f = 0 \quad \mbox{in} \ V \times \Rd
\end{equation*}
provided that, for all $g\in L^2(V; H^1_\gamma)$
\begin{equation}
\label{e.weaksoldef1}
\int_{V\times \Rd} 
\nabla_vg\cdot \a\nabla_v f \,dm_t
=
\int_{V\times \Rd} 
g\left(
v\cdot \nabla_xf
- \partial_t f
- \nabla H\cdot \nabla_vf \right)\,dm_t\,.
\end{equation}

Now that we have defined the major function space we will be working with, we state a few inequalities satisfied by elements of this space. 
\begin{prop}[{Poincar\'e inequality for $H^1_\kin$~\cite[Proposition 6.2]{AAMN}}]
\label{p.hypoelliptic.poincare}
Fix a bounded $C^{1,1}$ domain~$V\subseteq\R\times\Rd$. There exists~$C(d,V)<\infty$ such that, for every $f \in H^1_{\kin} (V)$, 
\begin{equation*} 
\left\| f - (f)_V \right\|_{L^2 (V; L^2_\gamma )} 
\leq
C \left\llbracket f \right\rrbracket_{H^1_\kin(V)}\,.
\end{equation*}
Moreover, if in addition we have $f\in H^1_{\kin,0}(V)$, then
\begin{equation*} 
\left\| f \right\|_{L^2(V;L^2_\gamma)} 
\leq
C \left\llbracket f \right\rrbracket_{H^1_\kin(V)}\,.
\end{equation*}
\end{prop}

As we are mostly focused on studying solutions to \eref{hypoeq} in all of $\Rd$, we make use of the following extension of \pref{hypoelliptic.poincare} that applies independently of the domain.
\begin{prop}[Nash Inequality for $H^1_\kin$]
\label{p.nashineqkin}
For $f\in H^1_\kin(\Rd)$, there exists $C(d)<\infty$ such that, for $t\geq 0$
$$\int_{(t, t+1)\times\Rd\times\Rd} f^2 \, dm_t \leq C\max\insb*{\snorm{f}_{H^1_\kin((t,t+1)\times\Rd)}^{\frac{2d}{d+2}}\norm{f}_{L^1((t,t+1)\times\Rd; L^1_\gamma)}^{\frac{4}{d+2}}\;,\; \snorm{f}_{H^1_\kin((t,t+1)\times\Rd)}^2} \, .$$
\end{prop}
\begin{proof}
For any $R>0$, we can apply \pref{hypoelliptic.poincare} see that 
\begin{align*}
\int_{(t,t+1)\times\Rd\times\Rd} f^2 \,dm_t &= \sum_{y\in R \Zd} \int_{(t,t+1)\times Q_{R}(y)\times\Rd} f^2 \,dm_t \\
&\leq \sum_{y\in R\Zd}C \norm{\nabla_v f}_{L^2((t,t+1)\times Q_R(y) ; L^2_\gamma)}^2 \\
&\qquad+ C(1+R^2) \norm{(\partial_t - v\cdot \nabla_x )f}^2_{L^2( (t,t+1)\times Q_R(y); H^{-1}_\gamma)}\\ 
&\qquad+ CR^d\inp*{\dashint_{(t, t+1)\times Q_R(y) \times \Rd} f\,dm_t}^2 \\
&\leq C(1+R^2) \snorm{f}_{H^1_\kin((t,t+1)\times \Rd)}^2 + CR^{-d}\norm{f}^2_{L^1((t,t+1)\times\Rd; L^1_\gamma)} \,.
\end{align*}
Now, to optimize this inequality we choose $R= \snorm{f}_{H^1_\kin((t,t+1)\times \Rd)}^{-\frac{2}{d+2}}\norm{f}_{L^1((t,t+1)\times\Rd; L^1_\gamma)}^\frac{2}{d+2}$. Then, if $R\geq 1$, the above inequality reduces to 
$$\int_{(t,t+1)\times\Rd\times\Rd} f^2 \,dm_t \leq C \snorm{f}_{H^1_\kin((t, t+1)\times\Rd)}^{\frac{2d}{d+2}} \norm{f}_{L^1((t,t+1)\times\Rd; L^1_\gamma)}^{\frac{4}{d+2}}.$$
On the other hand, if $R\leq1$, then $\norm{f}_{L^1((t,t+1)\times\Rd; L^1_\gamma)} \leq \snorm{f}_{H^1_\kin((t,t+1)\times \Rd)}$. Using this bound, we then see that
\begin{align*} 
\int_{(t,t+1)\times\Rd\times\Rd} f^2 \,dm_t &\leq C\snorm{f}_{H^1_\kin((t,t+1)\times \Rd)}^2 + C\snorm{f}_{H^1_\kin((t,t+1)\times \Rd)}^{\frac{2d}{d+2}}\norm{f}^{2- \frac{2d}{d+2}}_{L^1((t,t+1)\times\Rd; L^1_\gamma)} \\
&\leq C\snorm{f}_{H^1_\kin((t,t+1)\times \Rd)}^2\,.
\end{align*}
\end{proof}

\subsection{Short time function spaces and weak solutions}
While for longer times, the friction and diffusion terms operate on the same scale, this is not the case in shorter time with initial data supported in a small set. In particular, in order to study the short time behavior of the fundamental solution, we show that we want to bound solutions to the following equation:
\begin{equation}
\label{e.smallteq}
\begin{cases}
 \partial_t f - \nabla_v\cdot\a\nabla_v f + \ep^2v \cdot \a\nabla_v f - v \cdot \nabla_x f + \ep\nabla H\left(\ep^{3} x \right) \cdot \nabla_v f = 0
& \text{ in } \R_+\times \Rd \times\Rd \\
 f(0,\cdot) = g(x,v) \,.
\end{cases}
\end{equation}
where $g\in C^\infty(B_1\times B_1)$, $g\geq 0$, and $\int g\,dxdv = 1$. 

To motivate this, note that defining $f_\ep(t,x,v) = \ep^{-4d}f(\ep^{-2}t, \ep^{-3}x, \ep^{-1} v)$, we see that $f_\ep$ solves the following equation:

\begin{equation*}
\begin{cases}
\partial_t f_\ep - \nabla_v\cdot\a\nabla_v f_\ep + v \cdot \a\nabla_v f_\ep - v \cdot \nabla_x f_\ep + \nabla H\left(x \right) \cdot \nabla_v f_\ep = 0
& \text{ in } \R_+\times \Rd \times\Rd \\
 f_\ep(0,x,v) = \ep^{-4d}g(\ep^{-3} x, \ep^{-1} v) \,,
\end{cases}
\end{equation*}
which, after sending $\ep\to 0$, gives the fundamental solution of this equation. Therefore, to prove bounds on the fundamental solution for short times, it suffices to prove bounds on solutions to \eref{smallteq} that are independent of $\ep$. Additionally, this suggests that for small time, we should be treating the $v\cdot\a\nabla_v$ term and the $\nabla H \cdot \nabla_v$ terms as perturbative additions to the Kolmogorov equation, given by
\begin{equation*}
\begin{cases}
\partial_t f - \nabla_v\cdot\a\nabla_v f - v\cdot\nabla_x f = 0 & \text{in} \R_+\times\Rd\times\Rd\\
f(0,x,v) = g(x,v)\,.
\end{cases}
\end{equation*}
With this in mind, we devote the rest of this subsection to describing the alternate structure needed to describe these equations in this way.

We first define $\mathbf{H}^1(V)$ be $H^1(V)$ endowed with the norm
$$\norm{f}_{\mathbf{H}^1(V)} = \norm{f}_{L^2(V)} + \norm{vf}_{L^2(V)} + \norm{\nabla_v f}_{L^2(V)}\, ,$$
where $V$ is a subset of $\Rd$. We then define, for $I\subset \R$ and $U,V\subset \Rd$,  $\mc{H}^1(I\times U \times V)$ to be the following space 
$$\mathcal{H}^1(I\times U\times V) = \{ f\in L^2(I\times U; \mathbf{H}^1(V)) : \norm{\partial_t f - v\cdot\nabla_x f }_{L^2(I\times U; \mathbf{H}^{-1}(V))} < \infty\}\,,$$
where $\mathbf{H}^{-1}(V)$ is the dual space of $\mathbf{H}^1(V)$.

Mirroring the previous section, we now state a Poincar\'e inequality for the space $\mathbf{H}^1(V)$. The proof is very similar to the one given for \pref{hypoelliptic.poincare} in \cite{AAMN}, but we reproduce it here anyway for the sake of completeness.
\begin{prop}[Poincar\'e Inequality for $\mc{H}^1$]
\label{p.newpoincare}
There exists some $C(I,U, V)<\infty$ such that for $f\in \mc{H}^1(I\times U\times V)$,
$$\norm{f - (f)_{I\times U\times V}}_{L^2(I\times U\times V)} \leq C\inp*{\norm{\nabla_v f}_{L^2(I\times U\times V)} + \norm{\partial_t f -v\cdot\nabla_x f}_{L^2(I\times U ; \mathbf{H}^{-1}(V))}}\, .$$
Furthermore, we see that, for $f\in \mc{H}^1(I\times Q_R\times Q_S)$, 
\begin{align*}
\norm{f - (f)&_{I\times Q_R\times Q_S}}_{L^2(I\times Q_R\times Q_S)} \\
&\leq CS\norm{\nabla_v f}_{L^2(I\times Q_R\times Q_S)} + C\inp*{1+\tfrac{R}{S^2}}\norm{\partial_t f -v\cdot\nabla_x f}_{L^2(I\times Q_R; \mathbf{H}^{-1}(Q_S))}\, .
\end{align*}
\end{prop}

\begin{proof}
We first show that 
$$\norm{f- (f)_V}_{L^2(V)} \leq C(V)\norm{\nabla_vf}_{L^2(V)}\,.$$
For fixed $x,t$, we can apply the usual Poincar\'e inequality to see that 
$$\norm{f(t,x,\cdot) - (f(t,x,\cdot))_V}_{L^2(V)} \leq C(V)\norm{\nabla_v f}_{L^2(V)}.$$
Then, integrating in $x$ and $t$ gives the desired bound.

Next, we take $\phi \in C^\infty_c(I\times U)$ such that 
$$\norm{\phi}_{L^2(I\times U)} + \norm{\nabla_x \phi}_{L^2(I\times U)} + \norm{\partial_t \phi}_{L^2(I\times U)} \leq 1\,,$$
and we want to show that 
$$\abs*{\int_{I\times U}\partial_{x_i} \phi \inprod{f}_V } \leq C\inp*{ \norm{\nabla_v f}_{L^2(I\times U\times V)} + \norm{\partial_t f - v\cdot\nabla_x f}_{L^2(I\times U; \mathbf{H}^{-1}(V)}}\,.$$
We define $\xi_i\in C^\infty(V)$ so that 
$$\int_V \xi_i(v)\,dv = 0 \qquad \text{and } \qquad\int_{V} v \xi_i(v) \,dv = e_i \,.$$
Rearranging, we see that
\begin{align*}
\int_{I\times U} \partial_{x_i}\phi \inprod{f}_V &= -\int_{I\times U\times V} \xi_i(v) (\partial_t \phi - v\cdot\nabla_x \phi) \inprod{f}_V \\
&= -\int_{I\times U\times V} \xi_i(v)(\partial_t \phi -  v\cdot\nabla_x \phi) f + \int_{I \times U\times V} \xi_i(v) (\partial_t\phi - v\cdot\nabla_x \phi) (f- \inprod{f}_V) 
\,.
\end{align*}
Taking the first term, we can integrate by parts to see that
\begin{align*}
-\int_{I\times U\times V} \xi_i(v) (\partial_t \phi - v\cdot\nabla_x \phi) f &=  \int_{I\times U \times V} \xi_i(v) \phi (\partial_t f - v\cdot\nabla_x f) \\
&\leq \norm{\xi_i}_{\mathbf{H}^1(V)} \norm{\phi}_{L^2(U)} \norm{\partial_t f - v\cdot \nabla_x f}_{L^2(I\times U; \mathbf{H}^{-1}(V))} \\
&\leq \norm{\xi_i}_{\mathbf{H}^1(V)} \norm{\phi}_{H^1(I\times U))} \norm{\partial_tf - v\cdot \nabla_x f}_{L^2(I\times U; \mathbf{H}^{-1}(V))}
\,.
\end{align*}
Looking at the second term, we can again use the Poincar\'e inequality to see that
\begin{align*}
\int_{I \times U\times V} \xi_i(v) (&\partial_t \phi - v\cdot\nabla_x \phi) (f- \inprod{f}_V) \\
&\leq
(\norm{\partial_t \phi}_{L^2(I\times U)}\norm{\xi_i}_{L^2(V)} + \norm{\nabla_x \phi}_{L^2(I\times U)}\norm{v\xi_i}_{L^2(V)})  \norm{f - \inprod{f}_v}_{L^2(I\times U\times V)} \\
&\leq C(V)\norm{\xi_i}_{\mathbf{H}^1(V)} \norm{\phi}_{H^1(I\times U))} \norm{\nabla_v f}_{L^2(V)}\,.
\end{align*}
Additionally, repeating the same argument with $\xi_0$ such that 
$$\int_V \xi_0(v)\,dv = 1 \qquad \text{and } \qquad\int_{V} v \xi_0(v) \,dv = 0 \,,$$
we can see that for $\phi$ defined as above
$$\abs*{\int_{I\times U}\partial_{t} \phi \inprod{f}_V } \leq C\inp*{ \norm{\nabla_v f}_{L^2(I\times U\times V)} + \norm{\partial_t f - v\cdot\nabla_x f}_{L^2(I\times U; \mathbf{H}^{-1}(V)}}\,.$$
Therefore this shows that 
$$\norm{(f)_V - (f)_{I\times U\times V}}_{L^2(I\times U \times V)} \leq C\inp*{\norm{\nabla_v f}_{L^2(I\times U\times V)} + \norm{\partial_t f -v\cdot\nabla_x f}_{L^2(I\times U ; \mathbf{H}^{-1}(V))}}\,,$$
completing the proof of the desired bound.

For the scaling, we can take the result for $I\times Q_1 \times Q_1$ and then for $f$ defined on $I\times Q_R\times Q_S$, apply the unit ball result to $f(t,Rx, Sv)$. After enlarging the constant on the transport term, this yields the scaling stated above.
\end{proof}

Continuing, we next prove a version of the Nash inequality for this function space.
\begin{prop}[Nash Inequality for $\mc{H}^1$]
\label{p.newnash}
There exists $C(d)<\infty$ such that, for \\$f\in \mc{H}^1(\R\times\Rd\times\Rd)$ and $t\geq 0$,
$$\int_{(t, t+1)\times\Rd\times\Rd} f^2  \leq C \max\insb*{\snorm{f}_{\mc{H}^1((t,t+1)\times\Rd\times\Rd)}^{\frac{4d}{2d+1}}\norm{f}_{L^1((t,t+1)\times\Rd\times\Rd)}^{\frac{2}{2d+1}}\;,\; \snorm{f}_{\mc{H}^1((t,t+1)\times\Rd\times\Rd)}^2} \, .$$
\end{prop}
\begin{proof}
For any $R>0$, we can apply \pref{newpoincare} to see that 
\begin{align*}
\int_{(t,t+1)\times\Rd\times\Rd} f^2  &= \sum_{(y,w)\in R^3 \Zd\times R\Zd} \int_{(t,t+1)\times Q_{R^3}(y)\times Q_R(w)} f^2 \\
&\leq \sum_{(y,w)\in R^3 \Zd\times R\Zd} CR^2 \norm{\nabla_v f}_{L^2((t,t+1)\times Q_{R^3}(y)\times Q_R(w) )}^2 \\
&\qquad+ C(1+R^2) \norm*{\partial_tf - v\cdot \nabla_x f }^2_{L^2( (t,t+1)\times Q_{R^3}(y); H^{-1}(Q_R(w)))}\\ 
&\qquad+ CR^{4d}\inp*{\dashint_{(t, t+1)\times Q_{R^3}(y) \times Q_R(w)} f}^2 \\
&\leq C(1+R^2) \snorm{f}_{\mc{H}^1((t,t+1)\times \Rd\times\Rd)}^2 + CR^{-4d}\norm{f}^2_{L^1((t,t+1)\times\Rd\times\Rd)} .
\end{align*}
Now, to optimize this inequality we  choose $R= \snorm{f}_{\mc{H}^1((t,t+1)\times \Rd\times\Rd)}^{-\frac{1}{2d+1}}\norm{f}_{L^1((t,t+1)\times\Rd\times\Rd)}^\frac{1}{2d+1}$. Then, if $R\geq 1$, the above inequality reduces to 
$$\int_{(t,t+1)\times\Rd\times\Rd} f^2  \leq C \snorm{f}_{\mc{H}^1((t, t+1)\times\Rd\times\Rd)}^{\frac{4d}{2d+1}} \norm{f}_{L^1((t,t+1)\times\Rd\times\Rd)}^{\frac{2}{2d+1}}.$$
On the other hand, if $R\leq1$, then $\norm{f}_{L^1((t,t+1)\times\Rd\times\Rd)} \leq \snorm{f}_{\mc{H}^1((t,t+1)\times \Rd\times\Rd)}$. Using this bound, we then see that
\begin{align*} 
\int_{(t,t+1)\times\Rd\times\Rd} f^2  
& \leq C\snorm{f}_{\mc{H}^1((t,t+1)\times \Rd\times\Rd)}^2 + C\snorm{f}_{\mc{H}^1((t,t+1)\times \Rd\times\Rd)}^{\frac{4d}{2d+1}}\norm{f}^{2- \frac{4d}{2d+1}}_{L^1((t,t+1)\times\Rd\times\Rd)} \\
&\leq C\snorm{f}_{\mc{H}^1((t,t+1)\times \Rd\times\Rd)}^2\,.
\end{align*}
\end{proof}

\section{Nash-Aronson Estimate for the Kinetic Fokker-Planck Equation}
\subsection{Proof of the Large Time Bound}
We begin by proving \tref{nasharonson} in the case $t-s\geq 1$. To simplify notation, we take $s=0$ throughout this section. Here, we can derive bounds on the fundamental solution $P$ by deriving bounds on $f$, a solution to \eref{hypoeq}, with initial data $g$ such that $g\in C^\infty_c(B_1\times B_1)$, $g\geq 0$ and $\int_{\Rd\times \Rd} g = 1$. 

We first observe that
\begin{align}
\label{e.presmass}
\partial_t\int_{\Rd\times\Rd} &\abs{f(t, \cdot)}\,dm \\
&= \int_{\Rd\times\Rd} \partial_t f(t, \cdot)\,dm \nonumber\\
&=- \int_{\Rd} \nabla_v\cdot\a\nabla_v f(t, \cdot) +v\cdot\a\nabla_v f(t,\cdot) - v\cdot\nabla_xf(t,\cdot) + \nabla H\cdot\nabla_vf(t,\cdot)\,dm \nonumber \\
&= 0. \nonumber
\end{align}
From this, we can conclude that $\int_{\Rd\times \Rd} \abs{f(t, \cdot)} = \int_{\Rd\times \Rd} g = 1$. 

Next, we will show that, for some $C(d)<\infty$
\begin{equation}
\label{e.diagbound}
\int_{\Rd\times\Rd}\abs{f(t, \cdot)}^2\,dm \leq Ct^{-\frac{d}{2}} \,,
\end{equation}
when $t\geq 1$, which we call the diagonal estimate. Since $f$ is a solution to \eref{hypoeq}, it follows from \eref{weaksoldef1} that
$$\norm{\partial_t f - v\cdot\nabla_x f}_{L^2((s,t)\times \Rd; H^{-1}_\gamma)} \leq C\norm{\nabla_v f}_{L^2((s,t)\times\Rd; L^2_\gamma)}\, .$$
Then, using this bound we can compute that, for $t>s\geq 0$,
\begin{align*}
\int_{\Rd\times\Rd}( f(t,\cdot)^2-f(s,\cdot)^2)\,dm &= -2\int_{(s,t)\times\Rd\times\Rd} \nabla_v f\cdot\a\nabla_v f +f (-v\cdot\nabla_x f +\nabla H\cdot\nabla_v f)\,dm \\
&\leq -2\int_{(s,t)\times\Rd\times\Rd} \abs{\nabla_v f}^2 \,dm \\
&\leq -C \snorm{f}^2_{H^1_{\kin}((s,t)\times\Rd)}
\,, 
\end{align*}
where in the first line, we use integration by parts to see that
$$\int_{(s,t)\times\Rd\times\Rd}f (-v\cdot\nabla_x f +\nabla H\cdot\nabla_v f)\,dm =0\,.$$
So, $\norm{f(t,\cdot)}_{L^2(\Rd; L^2_\gamma)}$ is non-increasing. Then, letting $(s,t) = (t,t+1)$ for some $t\geq0$, we can use this bound, \pref{nashineqkin}, and \eref{presmass} to see that
\begin{align*}
\int_{\Rd\times\Rd} f(t+1, \cdot)^2 &- f(t,\cdot)^2 \,dm \\
&\leq -C \min\insb*{\norm{f}_{L^2((t,t+1)\times\Rd; L^2_\gamma)}^2 \; , \; \norm{f}_{L^2((t,t+1)\times\Rd; L^2_\gamma)}^{2+\frac{4}{d}}\norm{f}_{L^1((t,t+1)\times\Rd; L^1_\gamma)}^{-\frac{4}{d}}} \\
&\leq -C\min\insb*{\norm{f(t+1,\cdot)}_{L^2(\Rd; L^2_\gamma)}^2 \; , \; \norm{f(t+1,\cdot)}_{L^2(\Rd; L^2_\gamma)}^{2+\frac{4}{d}}} 
\,.
\end{align*}
Now, if $\norm{f(t+1,\cdot)}_{L^2(\Rd; L^2_\gamma)}\geq 1$, then
$$\int_{\Rd\times\Rd} f(t+1, \cdot)^2 - f(t,\cdot)^2 \,dm \leq -C\norm{f(t+1,\cdot)}_{L^2(\Rd; L^2_\gamma)}^2 \,.$$
An induction argument then implies that, for $t\in \N$ such that the above inequality holds, 
$$\norm{f(t,\cdot)}_{L^2(\Rd; L^2_\gamma)} \leq Ce^{-Ct}\,.$$
However, as $\norm{f(t, \cdot)}_{L^2(\Rd;L^2_\gamma)}$ is non-increasing, this last inequality can only hold for finite time. Therefore, the long time behavior will be given by considering the next case, that is when $\norm{f(t+1,\cdot)}_{L^2(\Rd; L^2_\gamma)}\leq 1$. 

In this case, we have that
$$\int_{\Rd\times\Rd} f(t+1, \cdot)^2 - f(t,\cdot)^2 \,dm \leq -C\norm{f(t+1,\cdot)}_{L^2(\Rd; L^2_\gamma)}^{2+\frac{4}{d}} \,.$$
We claim that this implies that 
$$\int_{\Rd\times\Rd} f(t, \cdot)^2 \,dxdv \leq Ct^{-\frac{d}{2}}\, .$$
To see this, we define $z_n = \norm{f(n,\cdot)}_{L^2(\Rd\times\Rd)}^{-\frac{4}{d}}$. Then, the above inequality says that
$$ z_{n+1}^{-\sfrac{d}{2}} - z_n^{-\sfrac{d}{2}} \leq -C z_{n+1}^{-1-\sfrac{d}{2}}\,.$$
Therefore, using this inequality, we see that for some $s\in [z_n, z_{n+1}]$, 
\begin{align}
\label{e.bigdumbiter}
z_{n+1} - z_n &= -C(z_{n+1}^{-\sfrac{d}{2}} - z_n^{-\sfrac{d}{2}})s^{1+\sfrac{d}{2}} \geq C \inp*{\frac{z_n}{z_{n+1}}}^{1+\sfrac{d}{2}}\,. 
\end{align}
Notice that, provided $z_{n+1} \leq 2z_n$, we can iterate this inequality as in the first case to get the desired result. As this is not necessarily the case, we need to be a bit more careful with this iteration. We define, for $t\in \N$
$$G_n = \insb*{k\in \N : k\leq n, \quad z_{k} \geq 2z_{k-1}}$$
and $B_n = \{k\in\N : k\leq n\} \setminus G_n$, we claim that, for all $n\in\N$,
\begin{equation}
\label{e.fixiter}
 z_n \geq C \min\{\abs{B_n}, 2^{\abs{G_n}}\}\,.
 \end{equation}
Suppose that \eref{fixiter} holds for some $n\in\N$. Then, it must be the case that either $G_{n+1} = \{n+1\}\cup G_n$ or $G_{n+1} = G_n$. In this first case, \eref{fixiter} holds for $z_{n+1}$ immediately from the definition of $G_n$ and the fact that $z_n$ is increasing. Otherwise, $B_{n+1} = \{n+1\}\cup B_n$, in which case, we use \eref{bigdumbiter} to see that $z_{n+1}\geq z_n + C$, which implies that $z_{n+1} \geq C\abs{B_{n+1}}$. Combining these two cases then gives \eref{fixiter}.

Furthermore, as $2^x \geq x$ for $x\in\R$, \eref{fixiter} reduces to 
$$z_n \geq C n\,$$
which implies that 
$$\int_{\Rd\times\Rd} f(n, \cdot)^2 \,dxdv \leq Cn^{-\frac{d}{2}}\, .$$

Finally, we can again use that $\norm{f(t,\cdot)}_{L^2(\Rd; L^2_\gamma)}$ is non increasing in $t$ to conclude that these bounds also hold for non-integer $t$. In particular, after potentially enlarging $C$,  we can combine the above estimates to see that \eref{diagbound} holds.

Next, we claim that there exists $C<\infty$ such that for $t\geq 1$ for all $A> 1$, 
\begin{equation}
\label{e.offdiagbound}
\int_{\Rd\times\Rd} \exp\inp*{\frac{\abs{x}^2}{2At}} f(t,\cdot)^2\,dm \leq C\exp\inp*{\frac{A}{A-1}} t^{-\frac{d}{2}}\,,
\end{equation}
Taking $\psi(t,x)$ to be a cutoff function, we can use the equation for $f$ to see that
\begin{equation}
\label{e.testfuncplugin}
\partial_t\int_{\Rd\times\Rd} \psi^2 f^2\,dm = 2\int_{\Rd\times\Rd} f^2 \psi\partial_t\psi- \psi^2\nabla_vf\cdot\a\nabla_v f  +\psi^2 f (v\cdot\nabla_xf -\nabla H\cdot\nabla_vf)\,dm\,.
\end{equation}

Integrating by parts and using Cauchy's inequality, we can see that 
\begin{align*}
\int_{\Rd\times\Rd} \psi^2 f (-v\cdot\nabla_xf + \nabla H \cdot\nabla_vf) \,dm &= \int_{\Rd\times\Rd} f^2 \psi (v\cdot\nabla \psi )\,dm \\
&= 2 \int_{\Rd\times\Rd} f\psi \nabla\psi\cdot\nabla_v f \,dm \\
&\geq - \int_{\Rd\times\Rd} f^2\abs{\nabla \psi}^2 +  \psi^2\abs{\nabla_v f}^2\,dm  \, .
\end{align*}
Then, using this bound and the ellipticity of $\a$ in \eref{testfuncplugin}, we see that 
\begin{equation}
\label{e.testfuncbound}
\partial_t\int_{\Rd\times\Rd} \psi^2 f^2\,dm \leq 2\int_{\Rd\times\Rd}  f^2\inp*{ \psi\partial_t\psi  + \abs{\nabla \psi}^2 } \,dm \, .
\end{equation}

Now, we take $\psi(t,x) = \exp\inp*{\frac{\abs{x}^2}{4At}}$, and compute that
$$\psi\partial_t\psi +\abs{\nabla \psi}^2 = \psi^2\inp*{-\frac{\abs{x}^2}{4At^2}  + \frac{\abs{x}^2}{4A^2t^2}} = -\psi^2\inp*{1-\frac{1}{A}}\frac{\abs{x}^2}{4At^2}.$$
Feeding this into \eref{testfuncbound}, we now have that
$$\partial_t\int_{\Rd\times\Rd} \psi^2 f^2\,dm \leq -\frac{2}{t}\inp*{1-\frac{1}{A}}\int_{\Rd\times\Rd} \inp*{\frac{\abs{x}^2}{4At}}\psi^2f^2\,dm.$$
Applying \eref{diagbound}, we see that, for any $M>0$,
\begin{align*}
\int_{\insb*{\abs{x}^2 <(4At)M}\times\Rd} \psi^2 f^2\,dm &\leq e^{2M} \int_{\insb*{\abs{x}^2 <(4At)M}\times\Rd} f^2\,dm \leq e^{2M} \int_{\Rd\times\Rd} f^2\,dm \leq C e^{2M}t^{-\sfrac{d}{2}}\,.
\end{align*}
Using this estimate, and noting that the integral above is also positive, we see that 
\begin{align*}
\partial_t \int_{\Rd\times\Rd} \psi^2 f^2\,dm &\leq -\frac{2}{t}\inp*{1-\frac{1}{A}}\int_{\Rd\times\Rd} \inp*{\frac{\abs{x}^2}{4At}} \psi^2 f^2\,dm \\
&\leq -\frac{2}{t}\inp*{1-\frac{1}{A}}\int_{\insb*{\abs{x}^2 > 4AtM}\times\Rd} \inp*{\frac{\abs{x}^2}{4At}} \psi^2 f^2\,dm \\
&\leq -\frac{2M}{t}\inp*{1-\frac{1}{A}}\int_{\insb*{\abs{x}^2 > 4AtM}\times\Rd}  \psi^2 f^2\,dm \\
&= -\frac{2M}{t}\inp*{1-\frac{1}{A}}\!\int_{\Rd\times\Rd}  \psi^2 f^2\,dm +\frac{2M}{t}\inp*{1-\frac{1}{A}}\!\int_{\insb*{\abs{x}^2 < 4AtM}\times\Rd}  \psi^2 f^2 \,dm\\
&\leq-\frac{2M}{t}\inp*{1-\frac{1}{A}}\!\int_{\Rd\times\Rd}  \psi^2 f^2\,dm + C\inp*{1-\frac{1}{A}} Me^{2M}t^{-\frac{d}{2}-1}.
\end{align*}
In particular, if $M$ is taken so that $2M\inp*{1-\frac{1}{A}} = 1 + \frac{d}{2}$, we can use the above inequality to compute that
\begin{align*}
\partial_t&\inp*{t^{1+\sfrac{d}{2}} \int_{\Rd\times\Rd} \psi^2 f^2\,dm} \\
&\quad= t^{1+\sfrac{d}{2}}\inp*{\partial_t \int_{\Rd\times\Rd} \psi^2 f^2\,dm + \frac{1}{t}\inp*{1+\frac{d}{2}} \int_{\Rd\times\Rd} \psi^2f^2\,dm} \\
&\quad\leq t^{1+\sfrac{d}{2}}\inp*{-\frac{1}{t}\inp*{1+\frac{d}{2}}\int_{\Rd\times\Rd}  \psi^2 f^2\,dm + \frac{1}{t}\inp*{1+\frac{d}{2}} \int_{\Rd\times\Rd} \psi^2f^2\,dm} + C\inp*{1-\frac{1}{A}} Me^{2M} \\
&\quad\leq C\inp*{1-\frac{1}{A}} Me^{2M} \\
&\quad= C e^{\frac{CA}{A-1}} \,.
\end{align*}
From here, we can integrate in time to get the desired bound. 

To conclude, we can use \eref{offdiagbound} to see that, for $t\geq 1$ and $x,y,v,w\in\Rd$,
\begin{align*}
P(2t,x,v, y, w) 
&= \int_{\Rd\times\Rd} P(t,x,v,z,u) P(t,z,u, y,w) \, dm(z,u) \\
&\leq Ce^{-\frac{\abs{x-y}^2}{8At}} \int_{\Rd\times\Rd} e^{\frac{\abs{x-z}^2}{4At}} e^{\frac{\abs{y-z}^2}{4At}} P(t,x,v,z,u) P(t,z,u, y,w) \, dm(z,u) \\
&\leq Ce^{-\frac{\abs{x-y}^2}{8At}} \inp*{\int_{\Rd\times\Rd} e^{\frac{\abs{x-z}^2}{2At}}P(t,x,v,z,u) ^2\,dm(z,u)}^\frac{1}{2} \\
&\qquad\qquad\qquad\qquad \times\inp*{\int_{\Rd\times\Rd} e^{\frac{\abs{y-z}^2}{2At}}P(t,z,u, y,w) ^2\,dm(z,u)}^\frac{1}{2} \\
&\leq C t^{-\frac{d}{2}}e^{-\frac{\abs{x-y}^2}{8At}}\, .
\end{align*}

\subsection{Proof of the Small Time Bound}
We now finish by proving \tref{nasharonson} when $t-s\leq 1$. Again, we set $s=0$ for this subsection. Following the discussion in section $2.2$, we consider $f$ as a solution to the following equation:
\begin{equation}
\label{e.scaledeq}
\begin{cases}
 \partial_t f - \nabla_v\cdot\a\nabla_v f + \ep^2v \cdot \a\nabla_v f - v \cdot \nabla_x f + \ep\nabla H\left(\ep^{3} x \right) \cdot \nabla_v f = 0
& \text{ in } \R_+\times \Rd \times\Rd \\
 f(0,\cdot) = g(x,v) \,.
\end{cases}
\end{equation}
where $g\in C^\infty(B_1\times B_1)$, $g\geq 0$, and $\int g\,dxdv = 1$, and repeat the argument given in the previous subsection.

To begin, we compute that, 
\begin{align*}
\partial_t\int_{\Rd\times\Rd}& \abs{f(t, \cdot)}\,dxdv \\
&= \int_{\Rd\times\Rd} \partial_t f(t, \cdot)\,dxdv \\
&= \int_{\Rd\times\Rd} \nabla_v\cdot\a\nabla_v f(t, \cdot) + v\cdot\nabla_xf(t,\cdot) - \ep^2 v\cdot\a\nabla_v f - \ep\nabla H_\ep(x) \cdot\nabla_v f\,dxdv \\
&\leq \ep^2\int_{\Rd\times\Rd} f\,dxdv\, .
\end{align*}
So, since $\int_{\Rd} g = 1$, we see that
\begin{equation}
\label{e.almostpresmass}
\int_{\Rd} \abs{f(t, \cdot)} = e^{C\ep^2 t}\,.
\end{equation}

Next, we will show that, for some $C(d)<\infty$
\begin{equation}
\label{e.smalltdiag}
\int_{\Rd\times\Rd}\abs{f(t, \cdot)}^2\,dxdv \leq Ct^{-2d} \,.
\end{equation}
 First, as $f$ is a solution to \eref{scaledeq}, we see that
\begin{align*}
\norm{\partial_t f - v\cdot\nabla_x f}_{L^2((s,t)\times \Rd; \mathbf{H}^{-1})}
&\leq \inp*{1+C\ep}\norm{\nabla_v f}_{L^2((s,t)\times\Rd\times\Rd)}\,.
\end{align*}

Then, using this bound we can compute that, for $t>s\geq 0$,
\begin{align*}
\int_{\Rd\times\Rd}  &f(t,\cdot)^2-f(s,\cdot)^2\,dxdv \\
&\leq -2\int_{(s,t)\times\Rd\times\Rd} \abs{\nabla_v f}^2 -f (v\cdot\nabla_xf) + \ep^2 f(v\cdot\a\nabla_v f) + \ep f\nabla H_\ep(x) \cdot\nabla_v f\,dtdxdv \\
&\leq -2\int_{(s,t)\times\Rd\times\Rd} \abs{\nabla_v f}^2 + C\ep^2 \int_{(s,t)\times\Rd\times\Rd}f^2\,dtdxdv \\
&\leq - C\snorm{f}^2_{\mc{H}^1((s,t)\times\Rd)},
\end{align*}
where the final inequality holds by applying \pref{newnash} and taking $\ep$ to be sufficiently small.

So, $\norm{f(t,\cdot)}_{L^2(\Rd\times\Rd)}$ is non-increasing. Then, letting $(s,t) = (t,t+1)$ for some $t\geq0$, we can use this bound, \pref{newnash}, and \eref{almostpresmass} to see that
\begin{align*}
\int_{\Rd\times\Rd} f(t+1, \cdot)^2 &- f(t,\cdot)^2 \,dxdv \\
&\leq -C \min\insb*{\norm{f}_{L^2((t,t+1)\times\Rd\times\Rd)}^2 \; , \; \norm{f}_{L^2((t,t+1)\times\Rd\times\Rd)}^{2+\frac{1}{d}}\norm{f}_{L^1((t,t+1)\times\Rd \times\Rd)}^{-\frac{1}{d}}} \\
&\leq -C\min\insb*{\norm{f(t+1,\cdot)}_{L^2(\Rd\times\Rd)}^2 \; , \; e^{-\sfrac{C\ep^2t}{d}}\norm{f(t+1,\cdot)}_{L^2(\Rd\times\Rd)}^{2+\frac{1}{d}}} \,.
\end{align*}
Now, if $\norm{f(t+1,\cdot)}_{L^2(\Rd\times\Rd)}\geq 1$, then
$$\int_{\Rd\times\Rd} f(t+1, \cdot)^2 - f(t,\cdot)^2 \,dxdv \leq -C\norm{f(t+1,\cdot)}_{L^2(\Rd\times\Rd)}^2 \,.$$
An induction argument then implies that, for $t\in \N$ such that the above inequality holds, 
$$\norm{f(t,\cdot)}_{L^2(\Rd\times\Rd)} \leq Ce^{-Ct}\,.$$
Similarly, if $\norm{f(t+1,\cdot)}_{L^2(\Rd\times\Rd)}\leq 1$, we have that
$$\int_{\Rd\times\Rd} f(t+1, \cdot)^2 - f(t,\cdot)^2 \,dxdv \leq -C\norm{f(t+1,\cdot)}_{L^2(\Rd\times\Rd)}^{2+\frac{1}{d}} \,.$$
We claim that this implies that 
$$\int_{\Rd\times\Rd} f(t+1, \cdot)^2 \,dxdv \leq Ct^{-2d}\, .$$
To see this, we define $z_n = \norm{f(n,\cdot)}_{L^2(\Rd\times\Rd)}^{-\frac{1}{d}}$ and note the previous inequality implies that
$$z_{n+1}^{-2d}-z_n^{-2d} \leq -Cz_{n+1}^{-2d-1}\,.$$
Therefore, for some $s\in [z_n, z_{n+1}]$
\begin{align}
\label{e.dumbiter}
z_{n+1} - z_n &= -C(z_{n+1}^{-2d} - z_n^{-2d})s^{2d+1} \\
&\geq C \inp*{\frac{z_n}{z_{n+1}}}^{2d+1}\,. \nonumber 
\end{align}
Notice that, provided $z_{n+1} \leq 2z_n$, we can iterate this inequality to get the desired result. As this may not be the case, we need to be a bit more careful with this iteration. We define, for $n\in \N$
$$G_n = \insb*{k\leq n : z_{k} \geq 2z_{k-1}}$$
and $B_n = \{k\leq n\} \setminus G_n$, we claim that 
$$ z_n \geq C \min\{\abs{B_n}, 2^{\abs{G_n}}\}.$$
To see this, suppose that the above bound holds for some $n\in\N$. Then, we see that either $G_{n+1} = \{n+1\}\cup G_n$ or $G_{n+1} = G_n$. In the case that $(n+1) \in G_{n+1}$, the bound immediately follows from the definition of $G_n$ and the fact that $z_n$ is increasing. Otherwise, $(n+1)\in B_{n+1}$, in which case, we use \eref{dumbiter} to see that $z_{n+1}\geq z_n + C$, which implies that $z_{n+1} \geq C\abs{B_{n+1}}$. Furthermore, as $2^x \geq x$ for $x\in\R$, we have shown that
$$z_n \geq C n\,,$$
which implies that 
$$\int_{\Rd\times\Rd} f(n, \cdot)^2 \,dxdv \leq Cn^{-2d}\, .$$

Finally, we can again use that $\norm{f(t,\cdot)}_{L^2(\Rd\times\Rd)}$ is non increasing in $t$ to conclude that these bounds also hold for non-integer $t$. In particular, after potentially enlarging $C$,  we can combine the above estimates to see that \eref{smalltdiag} holds.

Now, we claim that there exists $C<\infty$ such that for all $A> \frac{256(4+2\norm{\nabla H}^2_\infty)}{3}$ and $t\leq 1$, 
\begin{align}
\label{e.smalltoffdiag}
\int_{\Rd\times\Rd} &\exp\inp*{ \frac{3\abs{x+tw}^2}{At^3} + \frac{32\abs{v-w}^2}{3At} + \frac{32\abs{w}^2}{3At}} f^2\,dxdv \\
&\qquad \qquad \qquad \qquad\leq C \exp\inp*{\frac{CA}{3A-256(4+2\norm{\nabla H}_\infty^2)}}t^{-2d}\,. \nonumber
\end{align}
Taking $\psi(t,x,v)$ to be a cutoff function, we can use \eref{scaledeq} to see that
\begin{align*}
\partial_t\int_{\Rd\times\Rd} \psi^2 f^2\,dxdv 
&= 2\int_{\Rd\times\Rd} f^2 \psi\partial_t\psi- \nabla_v(\psi^2f)\cdot\a\nabla_v f  +\psi^2 f v\cdot\nabla_xf \\
&\qquad \qquad - \ep^2\psi^2 f(v\cdot\a\nabla_v f) - \ep \psi^2 f \nabla H_\ep(x) \cdot\nabla_v f\,dxdv\,.
\end{align*}
Looking at this expression term by term, we can see that, after using the Cauchy inequality,
\begin{align*}
\int_{\Rd\times\Rd} \nabla_v(\psi^2 f)\cdot\a\nabla_v f\,dxdv &= \int_{\Rd\times\Rd }\psi^2 \nabla_v f\cdot\a\nabla_v f + 2\psi f \nabla_v\psi\cdot\a\nabla_v f \,dxdv \\
&\geq \int_{\Rd\times\Rd} \psi^2 \abs{\nabla_v f}^2 - f^2 \abs{\nabla_v \psi}^2 - \psi^2 \abs{\nabla_v f}^2\,dxdv \\
&\geq - \int_{\Rd\times\Rd} f^2 \abs{\nabla_v \psi}^2\,dxdv\,.
\end{align*}
Additionally, as $\psi$ and $f$ are sufficiently regular, we can integrate by parts to see that 
\begin{align*}
\int_{\Rd\times\Rd} \psi^2 f (v\cdot\nabla_x f - \ep \nabla H_\ep\cdot\nabla_v f)\,dxdv &= -\int_{\Rd\times\Rd} f^2 \psi (v\cdot\nabla_x\psi - \ep \nabla H_\ep \cdot\nabla_v \psi) \,dxdv \,.
\end{align*}
Finally, we note that, after again integrating by parts
\begin{align*}
\int_{\Rd\times\Rd} \ep^2 \psi^2f(v\cdot\a\nabla_v f) \,dxdv &= -\int_{\Rd\times\Rd} \ep^2 f^2\psi(v\cdot\a\nabla_v \psi) \,dxdv \\
&\qquad -\int_{\Rd\times\Rd} \frac{\ep^2}{2}(\tr(\a) + v\cdot\nabla_v\cdot(\a))f^2\psi^2 \,dxdv \\
&\geq -\int_{\Rd\times\Rd} \ep^2 f^2\psi(v\cdot\a\nabla_v \psi) \,dxdv -\int_{\Rd\times\Rd} \Lambda\ep^2f^2\psi^2 \,dxdv\,.
\end{align*}

Then, applying the bounds above and using \eref{coefbnd}, we see that 
\begin{align}
\label{e.testfuncvbound}
\partial_t\int_{\Rd\times\Rd} \psi^2 f^2\,dxdv \leq 2\int_{\Rd\times\Rd}  f^2\big ( &\psi\partial_t\psi  + \abs{\nabla_v \psi}^2 - \psi v\cdot\nabla_x\psi \\
&+ \Lambda\ep^2  \psi v\cdot\nabla_v \psi + \Lambda\ep^2\psi^2 + \ep\psi \nabla H_\ep \cdot\nabla_v \psi\big)\,dxdv \,. \nonumber
\end{align}

Now, let's take
$$\psi_w(t,x,v) = \exp\inp*{ \frac{3\abs{x+tw}^2}{At^3} + \frac{32\abs{v-w}^2}{3At} + \frac{32\abs{w}^2}{3At}}.$$ 
Computing from this, we see that 
\begin{align*}
\psi_w&\partial_t\psi_w  + \abs{\nabla_v \psi_w}^2 - \psi_w (v\cdot\nabla_x\psi_w) + \Lambda\ep^2 \psi_w v\cdot\nabla_v \psi_w + \ep^2\psi_w^2 + \ep\psi_w \nabla H_\ep \cdot\nabla_v \psi_w \\
&= \inp*{ -\frac{9}{At^4} \abs*{x+tw}^2 + \frac{6}{At^3} w\cdot \inp*{x + tw}- \frac{32\abs{v-w}^2}{3At^2} + \frac{4096\abs{v-w}^2}{9A^2 t^2}-\frac{6}{At^3} v\cdot \inp*{x+tw}} \psi_w^2 \\
&\qquad + \inp*{\frac{64\Lambda\ep^2}{3At}v\cdot(v-w) - \frac{32}{3At^2}\abs{w}^2 + \Lambda\ep^2 + \frac{64\ep}{3 At} \nabla H_\ep\cdot (v-w)}\psi_w^2\\
&= \inp*{-\frac{9}{At^4}\abs*{x+tw}^2 - \frac{6}{At^3} (v-w)\cdot \inp*{x+tw} -
\frac{32}{3At^4}\inp*{1-\frac{128}{3A}}t^2\abs{v-w}^2  }\psi_w^2  \\   
&\qquad + \inp*{\frac{64\Lambda\ep^2}{3At}v\cdot(v-w) - \frac{32}{3At^2}\abs{w}^2 + \Lambda\ep^2 + \frac{64\ep}{3At} \nabla H_\ep\cdot (v-w)}\psi_w^2\\
&= -\frac{1}{At^4}\inp*{\abs{3x+3tw + tv-tw}^2 - t^2\abs{v-w}^2 +
\frac{32}{3}\inp*{1-\frac{128}{3A}}t^2\abs{v-w}^2  }\psi_w^2 \\
&\qquad + \inp*{\frac{64\Lambda\ep^2}{3At}v\cdot(v-w) - \frac{32}{3At^2}\abs{w}^2 + \Lambda\ep^2 + \frac{64\ep}{3At} \nabla H_\ep\cdot (v-w)}\psi_w^2\\
&= -\frac{1}{At^4}\inp*{\abs{3x+2tw + tv}^2 +
\inp*{\frac{29}{3}-\frac{4096}{9A}}\abs{t(v-w)}^2  + \frac{32}{3}\abs{tw}^2}\psi_w^2 \\
&\qquad + \inp*{\frac{64\Lambda\ep^2}{3At}v\cdot(v-w) + \Lambda\ep^2 + \frac{64\ep}{3At} \nabla H_\ep\cdot (v-w)}\psi_w^2\,.
\end{align*}
Using the triangle inequality
$$\frac{3}{2} \abs{x + tw}^2 \leq \abs{3x+2tw+tv}^2 + \frac{1}{3}\abs{t(v-w)}^2 \,.$$
This implies that
$$
\frac{3}{2} \abs{x + tw}^2 + \inp*{\frac{28}{3}-\frac{4096}{9A}}\abs{t(v-w)}^2 
\leq
\abs{3x+2tw + tv}^2 +
\inp*{\frac{29}{3}-\frac{4096}{9A}}\abs{t(v-w)}^2   \, .
$$
Likewise, we see that
$$\frac{64\ep}{3At} \nabla H_\ep\cdot (v-w) \leq \ep^2 + \frac{1024\norm{\nabla H}^2_\infty}{9A^2 t^4}\abs{t(v-w)}^2 \,.$$
Furthermore, as $\ep \leq (4\Lambda)^{-\frac{1}{2}}$ and $t\leq 1$, we see that
\begin{align*}
\frac{64\Lambda\ep^2}{3At}v\cdot(v-w) &\leq \frac{16}{3At^3}(tv)\cdot(t(v-w)) \\
&= \frac{16}{3At^3}\abs{t(v-w)}^2 + \frac{16}{3At^3} (tw)\cdot (t(v-w)) \\
&\leq \frac{8}{At^3}\abs{t(v-w)}^2 + \frac{8}{3At^3} \abs{tw}^2 \\
&\leq \frac{8}{At^4}\abs{t(v-w)}^2 + \frac{8}{3At^4} \abs{tw}^2
\,.
\end{align*}
Combining the above inequalities, we now have that
\begin{align*}
\psi_w&\partial_t\psi_w  + \abs{\nabla_v \psi_w}^2 - \psi_w (v\cdot\nabla_x\psi_w) + \Lambda\ep^2 \psi_w v\cdot\nabla_v \psi_w + \Lambda\ep^2\psi_w^2 + \ep\psi_w \nabla H_\ep \cdot\nabla_v \psi_w \\
&\leq -\frac{1}{At^4}\inp*{\frac{3}{2} \abs{x + tw}^2 + \frac{4}{3}\inp*{1-\frac{256(4+2\norm{\nabla H}_\infty^2)}{3A}}\abs{t(v-w)}^2   + 8\abs{tw}^2}\psi_w^2 + (\Lambda+1)\ep^{2}\psi_w^2\,.
\end{align*}

Feeding this into \eref{testfuncvbound}, we now have that
\begin{align*}
\partial_t&\int_{\Rd\times\Rd} \psi_w^2 f^2\,dxdv \\
&\leq -\frac{1}{At^4}\int_{\Rd\times\Rd}\inp*{\frac{3}{2} \abs{x + tw}^2 + \frac{4}{3}\inp*{1-\frac{256(4+2\norm{\nabla H}_\infty^2)}{3A}}\abs{t(v-w)}^2   + 8\abs{tw}^2}\psi_w^2 f^2\,dxdv\\
 &\quad+ (\Lambda+1)\ep^2\int_{\Rd\times\Rd}\psi_w^2f^2\,dxdv  \,.
\end{align*}
Now, to simplify notation, we let
$$ K_M := \insb*{9\abs{x+tw}^2+ 32\abs{t(v-w)}^2 + 32\abs{tw}^2 \leq (3At^3)M}\,.$$
Applying \eref{smalltdiag}, we see that, 
\begin{align*}
\int_{K_M} \psi_w^2 f^2\,dxdv &\leq e^{2M} \int_{K_M} f^2\,dxdv \leq C e^{2M}t^{-2d}\,.
\end{align*}
Using this estimate, and noting that the integral above is also positive, we see that 
\begin{align*}
\partial_t &\int_{\Rd\times\Rd} \psi_w^2 f^2\,dxdv - (\Lambda+1)\ep^2\int_{\Rd\times\Rd} \psi_w^2 f^2\,dxdv \\
&\leq -\frac{1}{At^4}\int_{\Rd\times\Rd}\inp*{\frac{3}{2} \abs{x + tw}^2 + \frac{4}{3}\inp*{1-\frac{256(4+2\norm{\nabla H}_\infty^2)}{3A}}\abs{t(v-w)}^2   + 8\abs{tw}^2}\psi_w^2 f^2\,dxdv \\
&\leq -\frac{1}{At^4}\int_{\comp{K_M}} \inp*{\frac{3}{2} \abs{x + tw}^2 + \frac{4}{3}\inp*{1-\frac{256(4+2\norm{\nabla H}_\infty^2)}{3A}}\abs{t(v-w)}^2   + 8\abs{tw}^2}\psi_w^2 f^2 \\
&\leq -\frac{1}{8t}\inp*{1-\frac{256(4+2\norm{\nabla H}_\infty^2)}{3A}}M\int_{\comp{K_M}} \psi_w^2 f^2\ \\
&= -\frac{1}{8t}\inp*{1-\frac{256(4+2\norm{\nabla H}_\infty^2)}{3A}}M\int  \psi_w^2 f^2 + \frac{1}{8t}\inp*{1-\frac{256(4+2\norm{\nabla H}_\infty^2)}{3A}}M\int_{K_M}  \psi_w^2 f^2 \\
&\leq-\frac{1}{8t}\inp*{1-\frac{256(4+2\norm{\nabla H}_\infty^2)}{3A}}M\int  \psi_w^2 f^2 + C\inp*{1-\frac{256(4+2\norm{\nabla H}_\infty^2)}{3A}}Me^{2M}t^{-2d-1}\, ,
\end{align*}
so, simplifying, and taking $\ep^2 \leq (4\Lambda)^{-1}\inp*{1- \frac{256(4+2\norm{\nabla H}_\infty^2)}{3A}}$,
\begin{align*}
\partial_t &\int_{\Rd\times\Rd} \psi_w^2 f^2\,dxdv \\
&\leq -\frac{1}{16t}\inp*{1-\frac{256(4+2\norm{\nabla H}_\infty^2)}{3A}}M\int  \psi_w^2 f^2 + C\inp*{1-\frac{256(4+2\norm{\nabla H}_\infty^2)}{3A}}Me^{2M}t^{-2d-1}
\,.
\end{align*}
Choosing $M$ so that 
$$C\inp*{1-\frac{256(4+2\norm{\nabla H}_\infty^2)}{3A}}M = 2d+2\, ,$$
we can compute that
\begin{align*}
\partial_t&\inp*{t^{2d+1} \int_{\Rd\times\Rd} \psi_w^2 f^2\,dxdv} \\
&\quad= t^{2d+1}\inp*{\partial_t \int_{\Rd\times\Rd} \psi_w^2 f^2\,dxdv + \inp*{2d+1} \frac{1}{t}\int_{\Rd\times\Rd} \psi_w^2f^2\,dxdv} \\
&\quad\leq t^{2d+1}\inp*{-\frac{1}{t}\inp*{2d+2}\int_{\Rd\times\Rd}  \psi_w^2 f^2\,dxdv} \\
&\qquad+ t^{2d+1}\inp*{C\inp*{1-\frac{256(4+2\norm{\nabla H}_\infty^2)}{3A}}Me^{2M}t^{-2d-1}+ \inp*{2d+1} \frac{1}{t}\int_{\Rd\times\Rd} \psi_w^2f^2\,dxdv} \\
&\quad\leq C\inp*{1-\frac{256(4+2\norm{\nabla H}_\infty^2)}{3A}}Me^{2M}\\
&\quad= C \exp\inp*{\frac{CA}{3A-256(4+2\norm{\nabla H}_\infty^2)}}
\,.
\end{align*}
From here, we can integrate in time to get the desired bound. 

To conclude, we first note that as the bounds on the left hand side of $\eref{smalltoffdiag}$ hold independently of $\ep$, we can send $\ep\to 0$ to see that these bounds also apply to $P$. Now, 
$$E_A(t,x,v,y,w) := \exp\inp*{\frac{3\abs{(x-y) +tw}^2}{At^3} + \frac{32\abs{v-w}^2 +32\abs{w}^2}{3At}}\,.$$
Furthermore, we can use the triangle inequality to see that
$$\abs{x-y+tw}^2 \leq 2\abs{z-y+tw}^2 + 4\abs{z-x+tv}^2 + 8t^2\abs{v}^2 +8t^2\abs{w}^2\,,$$
and 
$$\abs{v-w}^2 \leq 2\abs{v-u}^2 + 2\abs{u-w}^2\, .$$
Using these estimates as well as \eref{smalltoffdiag}, we see that, for $t\leq \frac{1}{4}$ and $x,y,v,w\in\Rd$,
\begin{align*}
&P(2t,x,v, y, w) \\
&= \int_{\Rd\times\Rd} P(t,x,v,z,u) P(t,z,u, y,w) \, dm(z,u) \\
&\leq CE_A(t,cx,cv,cy,cw)^{-1} \\
&\qquad\times\int E_A(t,z,u,x,v)P(t,x,v,z,u) E_A(t,z,u,y,w)P(t,z,u, y,w) \, dm(z,u) \\
&\leq CE_A(t,cx,cv,cy,cw)^{-1}\inp*{\int E_A^2P^2\,dm}^\frac{1}{2} \inp*{\int E_A^2P^2\,dm}^\frac{1}{2} \\
&\leq CE_A(t,cx,cv,cy,cw)^{-1}\inp*{\int_{\Rd\times\Rd} E_A^2P^2\,dzdu}^\frac{1}{2} \inp*{\int_{\Rd\times\Rd} E_A^2P^2\,dzdu}^\frac{1}{2} \\
&\leq C t^{-2d}\exp\inp*{-c\frac{\abs{(x-y) +tw}^2 + t^2\abs{v-w}^2 + t^2\abs{w}^2}{At^3}}\, . 
\end{align*}

\subsection*{Acknowledgments} 
The author was partially supported by NSF Grant DMS-2000200. 

\small
\bibliographystyle{abbrv}
\bibliography{hypoelliptic}
\end{document}